\documentstyle[12pt]{article}
\newtheorem{defi}{Definition}

\def\today{\ifcase\month\or
January\or February\or March\or April\or May\or June\or July\or August\or September\or October\or November\or December\fi
\space\number\day ,\number\year}

\newtheorem{Theorem}{Theorem}

\newtheorem{Lemma}{Lemma}
\newcommand{\bl}{\begin{Lemma}}

\newcommand{\el}{\end{Lemma}}
\newcommand{\be}{\begin{equation}}
\newcommand{\ee}{\end{equation}}
\newcommand{\bd}{\begin{defi}}
\newcommand{\ed}{\end{defi}}

\newtheorem{pro}{Proposition}
\newcommand{\bp}{\begin{pro}}
\newcommand{\ep}{\end{pro}}
\newcommand{\bt}{\begin{Theorem}}
\newcommand{\et}{\end{Theorem}}
\newtheorem{cor}{Corollary}
\newcommand{\bc}{\begin{cor}}
\newcommand{\ec}{\end{cor}}

\def\sqr#1#2{{\vcenter{\vbox{\hrule height.#2pt
\hbox{\vrule width.#2pt height#1pt \kern#1pt
\vrule width.#2pt}\hrule height.#2pt}}}}
\def\square{\mathchoice\sqr45\sqr45\sqr{2.1}3\sqr{1.5}3}
\newcommand{\qed}{\square}

 at 10truept

\begin{document}
\setlength{\textheight}{7.7truein}    
\setcounter{page}{1} \centerline{\bf Isoparametric functions and harmonic unit vector fields}
\centerline{\bf in K-Contact Geometry} \centerline{\bf } \baselineskip=13pt
\vspace*{10pt} \centerline{\footnotesize {\bf Philippe Rukimbira}} \baselineskip=12pt \centerline{\footnotesize\it
Department of Mathematics \& Statistics, Florida International University}
\baselineskip=10pt \centerline{\footnotesize\it Miami, Florida
33199, USA} \baselineskip=10pt \centerline{\footnotesize E-MAIL:
rukim@fiu.edu} \vspace*{0.225truein}

\vspace*{0.21truein} \abstract{\it } {We provide some examples of harmonic unit vector fields as normalized gradients of isoparametric functions coming from a K-contact geometry setting.}

\vskip 12pt
\noindent{MSC: 57C15, 53C57}{}{}

\vspace*{14pt}                  

\baselineskip=24pt
\section*{Introduction}
in \cite{BVA}, the authors showed that given an isoparametric function $f$ on an Einstein manifold, the normalized gradient vector field $\frac{\nabla f}{\|\nabla f\|}$ is a harmonic unit vector field.
In this paper, without the Einstein assumption, we present explicit isoparametric functions on double K-contact structures and also show that their normalized gradient vector fields are harmonic unit vector fields.
We may ask whether or not there are examples of double K-contact structures on non-Einstein manifolds. The answer to this question is negative if K-contact is replaced by Sasakian. A compact, double Sasakian manifold is of constant curvature 1 (see \cite{DRR}, \cite{TAY}).
\section{Transnormal functions}

A smooth function $f$ on $M$ is said to be {\bf transnormal} if there exists a real, smooth function $$b:\bf{R}\to\bf{R}^+$$ such that \begin{equation}\|\nabla f\|^2=b(f),\label{11}\end{equation} where
$\nabla f$ is the gradient vector field of $f$.

The function $f$ is said to be {\bf isoparametric} if there is another continuous function $$a:\bf{R}\to\bf{R}$$ such that the Laplacian of $f$,
$\Delta f$ satisfies:

\begin{equation}\Delta f=a(f).\label{22}\end{equation}

\bp Let $f$ be a transnormal function on a Riemannian manifold $(M,g)$. Then $N=\frac{\nabla f}{\|\nabla f\|}$ is a geodesic unit vector field defined on the complement of the critical set of $f$.
\ep

\begin{proof} Let $H$ be any vector field such that $g(\nabla f,H)=df(H)=0$.
Then, using transnormality,
\begin{eqnarray*}g(\nabla_{\nabla f}{\nabla f},H)&=\nabla f g(\nabla f,H)-g(\nabla f,\nabla_{\nabla f}H)\\&=-g(\nabla f,\nabla_H\nabla f)-g(\nabla f,[\nabla f,H])\\
&=-\frac{1}{2}Hg(\nabla f,\nabla f)-g(\nabla f,[\nabla f,H])\\&=-\frac{1}{2}b'(f)g(\nabla f,H)-g(\nabla f,[\nabla f,H])\\
&=-g(\nabla f,[\nabla f,H]) \label{1}
\end{eqnarray*}
\begin{equation}g(\nabla_{\nabla f}\nabla f,H)=-g(\nabla f,[\nabla f,H])\label{7}\end{equation}

But also, the Lie derivative of $df$ satisfies: $$L_{\nabla f}df (H)=\nabla fdf(H)-df([\nabla f,H])=-df([\nabla f,H])=-g(\nabla f,[\nabla f,H])$$ and $$L_{\nabla f}df(H)=(di_{\nabla f}df)(H)=H\|\nabla f\|^2=b'(f)df(H)=0.$$ We deduce from (\ref{1}) that $$\nabla_{\nabla f}\nabla f=K\nabla f$$ for some function $K$
on $M$. Hence
\begin{eqnarray*}\nabla_NN&=&\frac{1}{\|\nabla f\|}\nabla_{\nabla f}\frac{\nabla f}{\|\nabla f\|}\\&=&[\frac{1}{\|\nabla f\|}\nabla f(\frac{1}{\|\nabla f\|})+\frac{K}{\|\nabla f\|^2}]\nabla f\end{eqnarray*}

Since $\nabla_NN$ is orthogonal to $N$, and therefore to $\nabla f$, it follows that $\nabla_NN=0$.
$\qed $

\end{proof}

 \begin{Lemma}\label{l1}If $f$ is a transnormal function on a Riemannian manifold $(M,g)$, with $\|\nabla f\|^2=b(f)$, then the mean curvature $h$ of every regular level surface satisfies:
$$h=\frac{\Delta f}{\|\nabla f\|}+\frac{b'(f)}{2\sqrt b}.$$\end{Lemma}

\begin{proof} Let $E_i, i=1,2,...,m-1, N=\frac{\nabla f}{\|\nabla f\|}$ be an adapted orthonormal frame field, where $E_i\perp \nabla f$.
{\small
\begin{eqnarray*}h&=&-\sum_{i=1}^{m-1}g(\nabla_{E_i}N,E_i)\\
&=&-\sum_{i=1}^{m-1}g(\nabla_{E_i}N,E_i)-g(\nabla_NN,N)\\&=&-\frac{1}{\|\nabla f\|}(\sum_{i=1}^{m-1}g(\nabla_{E_i}\nabla f,E_i)+g(\nabla_N\nabla f,N))+\\&&\frac{N(\|\nabla f\|)}{\|\nabla f\|^2}g(\nabla f,N)\\
&=&\frac{1}{\|\nabla f\|}\Delta f+\frac{b'(f)}{2\sqrt b}.
\end{eqnarray*}} \hfill$\qed$
\end{proof}

\section{Harmonic unit vector fields}
Let $(M,g)$ be an m-dimensional Riemannian manifold. A unit vector field $Z$ on $M$ can be regarded as an immersion $Z\colon M\to T_1M$ of $M$ into its unit tangent bundle, which is itself a Riemannian manifold with its Sasaki metric $g_S$. In this setting, the induced metric on $M$ is given by$$Z^\ast g_S(X,Y)=g(X,Y)+g(\nabla_XZ,\nabla_YZ).$$
Denote by $A_Z$ and $L_Z$ the operators $$A_ZX=-\nabla_XZ$$ and $$L_ZX=X+A_Z^t(A_ZX).$$
The energy $E(Z)$ is given by $$E(Z)=\frac{1}{2}\int_Mtr L_ZdV_g=\frac{m}{2}Vol(M)+\frac{1}{2}\int_M\|\nabla Z\|^2dV_g$$ where $dV_g$ is the Riemannian volume form on $M$.
A critical point for the functional $E$ is called a harmonic unit vector field.

The critical point condition for $E$ have been derived in \cite{WIE}. More precisely,
$Z$ is a harmonic unit vector field on $(M,g)$ if and only if the one form $\nu_Z$, $$\nu_Z(X)=tr(u\mapsto (\nabla_uA_Z^t)X)$$ vanishes on $Z^\perp$.
Equivalently, the critical point condition for harmonic unit vector fields is again
\begin{equation}\sum_{u_i}g((\nabla_{u_i}A^t_Z)X,u_i)=0,~\forall X\in Z^\perp\label{cch}\end{equation} where the $u_i$s form an orthonormal basis.

Let $N$ be a geodesic vector field with integrable orthogonal complement $N^\perp$.
The endomorphism $A_N=-\nabla N$ is then symmetric. Indeed, for any $X,Y\perp N$,
\begin{eqnarray*}g(A_NX,Y)&=&g(-\nabla_XN,Y)\\
&=&-Xg(N,Y)+g(N,\nabla_XY)\\
&=&g(N,\nabla_YX)\\
&=&Yg(N,X)-g(\nabla_YN,X)\\
&=&g(X,A_NY)
\end{eqnarray*}

Since $A_NN=0$ and $g(A_NX,N)=0$, it follows that $A_N$ is symmetric.
Let $\lambda_i,~i=1,2,3,...m$ be the eigenvalues of $A_N$ on $N^\perp$. Let also $E_1,...,E_m$ be an orthonormal frame of $N^\perp$ consisting of eigenvectors. One has
 $$A_NN=A^t_NN=0,~A_NE_i=A^t_NE_i,~i=1,2,...,m$$  and $N$ is a harmonic unit vector field if and only if, for $j=1,...,n$
\begin{eqnarray}0=\nu_N(E_j)&=&\sum_{i=1}^mg((\nabla_{E_i}A^t_N)E_j,E_i)+g((\nabla_NA^t)E_j,N)
\end{eqnarray}
If $\tau$ is a field of symmetric endomorphisms, then so is $\nabla_E\tau$ for any $E$. We continue the above calculation
\begin{eqnarray*}
0&=&g((\nabla_{E_i}A_N)E_i,E_j)+g((\nabla_NA_N)N,E_j)\\
&=&\sum_{i=1}^mg(\nabla_{E_i}(A_NE_i),E_j)-g(A_N(\nabla_{E_i}E_i),E_j)\\
&=&\sum_{i=1}^mg(\nabla_{E_i}(\lambda_iE_i),E_j)-\sum_{i=1}^mg(\nabla_{E_i}E_i,\lambda_jE_j)\\
&=&E_j(\lambda_j)+\sum_{i=1}^m\lambda_ig(\nabla_{E_i}E_i,E_j)-\lambda_j\sum_{i=1}^mg(\nabla_{E_i}E_i,E_j)\\
&=&E_j(\lambda_j)+\sum_{i=1}^m(\lambda_i-\lambda_j)g(\nabla_{E_i}E_i,E_j)
\end{eqnarray*}
\begin{equation}0=E_j(\lambda_j)+\sum_{i=1}^m(\lambda_i-\lambda_j)g(\nabla_{E_i}E_i,E_j)\label{19}\end{equation}
On the other hand, Codazzi equations imply that
\begin{eqnarray*}g(R(E_i,E_j)E_i,N)&=&-g(R(E_i,E_j)N,E_i)\\
&=&g(\nabla_{E_i}(A_NE_j)-\nabla_{E_j}(A_NE_i)-A_N[E_i,E_j],E_i)\\
&=&E_ig(A_NE_j,E_i)-g(A_NE_j,\nabla_{E_i}E_i)-\\&&E_jg(A_NE_i,E_i)+g(A_NE_i,\nabla_{E_j}E_i)\\&&-g([E_i,E_j],A_NE_i)\\
&=&E_i(\lambda_jg(E_j,E_i)-\lambda_jg(E_j,\nabla_{E_i}E_i)-E_j(\lambda_i)+\\
&&\lambda_ig(E_i,\nabla_{E_j}E_i)-\lambda_ig([E_i,E_j],E_i)\\
&=&E_i(\lambda_j)g(E_j,E_i)+\lambda_jE_ig(E_j,E_i)-\\&&\lambda_jg(E_j,\nabla_{E_i}E_i)-
E_j(\lambda_i)-\lambda_ig([E_i,E_j],E_i)\\
&=&E_i(\lambda_j)\delta_{ij}+\lambda_jg(\nabla_{E_i}E_j,E_i)-E_j(\lambda_i)-\\&&\lambda_ig(\nabla_{E_i}E_j,E_i)
\end{eqnarray*} Summing over $i$ and using $g(R(N,E_j)N,N)=0$, we get:
\begin{eqnarray*}-\rho (E_j,N)&=&\sum_{i=1}^m(\lambda_j-\lambda_i)g(\nabla_{E_i}E_j,E_i)+E_j(\lambda_j)-E_j(\sum_{i=1}^m\lambda_i)\\
&=&\sum_{i=1}^m(\lambda_i-\lambda_j)g(\nabla_{E_i}E_i,E_j)+E_j(\lambda_j)-E_j(\sum_{i=1}^m\lambda_i)
\end{eqnarray*}
Now, combining with the previous identity (\ref{19}), one sees that
\begin{eqnarray*}-\rho (E_j,N)&=&0-E_j(\sum\lambda_i)
\end{eqnarray*}
\begin{equation}\rho (E_j,N)=E_j(\sum\lambda_i)\label{35}\end{equation}
Therefore, we see that $N$ is harmonic if and only if $$X(h)=\rho (X,N)$$ for all $X\perp N$, where $h=\sum_{i=1}^m\lambda_i$ is the mean curvature of $N^\perp$ and $\rho$ is the Ricci tensor.

For a geodesic vector field $N$, with integrable orthogonal complement, Identity (\ref{cch}) reduces to

\begin{equation}X(h)=\rho (X,N),~~\forall X\perp N\label{critical}\end{equation} where $h$ is the mean curvature of $N^\perp$ and $\rho$ is the Ricci tensor.
\section{Double K-Contact structures}

A contact metric structure on an odd-dimensional (2n+1) manifold $M$ is determined by the data of a 1-form $\alpha$ with Reeb field $Z$ together with a Riemannian metric $g$, called the adapted contact metric, and a partial complex structure $J$ such that the following identities hold:
\begin{itemize}
\item[i)] $\alpha\wedge (d\alpha )^n$ is a volume form on $M$.
\item[ii)] $J^2A=-A+\alpha (A)Z$
\item[iii)] $d\alpha (A,B)=2g(A,JB)$ for any tangent vectors $A$ and $B$.
\end{itemize}

If $Z$ is an infinitesimal isometry for $g$, then the structure is called K-contact. If in addition, the identity
$$(\nabla_A J)B=g(A,B)Z-\alpha (B)A$$ is satisfied for any two tangent vectors $A$ and $B$, then the structure is called Sasakian.
\bd
A double K-contact structure on a manifold $M$ is a pair of K-contact forms $\alpha$ and $\beta$ with same contact metric $g$ and
commuting Reeb vector fields $Z$ and $X$.
\ed

An example: On $\bf{S}^3\hookrightarrow\bf{R}^4$ with coordinates $x_1,y_1, x_2,y_2$, $x_1^2+y_1^2+x_2^2+y_2^2=1$. The standard K-contact form is $\alpha =y_1dx_1-x_1dy_1+y_2dx_2-x_2dy_2$ with Reeb field $Z=y_1\partial x_1-x_1\partial y_1+y_2\partial x_2-x_2\partial y_2$. Another K-contact form with same adapted metric is $\beta =-y_1dx_1+x_1dy_1+y_2dx_2-x_2dy_2$ with Reeb field $X=-y_1\partial x_1+x_1\partial y_1+y_2\partial x_2-x_2\partial y_2$.

$[X,Z]=0$ and the angle function $g(X,Z)=-y_1^2-x_1^2+y_2^2+x_2^2$ is isoparametric. Its gradient vector field is \begin{eqnarray*}2JX&=&2((-x_1\partial x_1-y_1\partial y_1+x_2\partial x_2+y_2\partial y_2)+\\&&2(x_1^2+y_1^2x_2^2-y_2^2)(x_1\partial x_1+y_1\partial y_1+x_2\partial x_2+y_2\partial y_2).\end{eqnarray*} $J$ is the standard partial complex structure on $\bf{S}^3$.

$N=\frac{2JX}{\|2JX\|}$ is a harmonic unit vector field as it will follow from results in the following sections. Similar examples as the above one can be repeated on any odd dimensional unit sphere.

\bp In the case of a double K-contact structure $(M,\alpha ,Z,\beta ,X)$, the angle function $f=g(X,Z)$ is always transnormal.\ep

 \begin{proof}Let's denote by $J$ and $\phi$ the respective complex structures on the contact sub-bundles. Then, the gradient of $f$ is given by
 $$\nabla f=2JX=2\phi Z.$$ Its norm square is therefore $$\|\nabla f\|^2=2\|JX\|^2=4(1-g(X,Z)^2)=4(1-f^2)=b(f)$$ with $b(t)=4(1-t^2).$ \hfill$\qed$ \end{proof}

 \begin{Lemma} The Lapacian of the transnormal function $f=g(Z,X)$ satisfies:
 $$\Delta f=(4n+4)f+2\sum_{i=1}^{2n-2}g(J\phi E_i,E_i),$$ where $E_i$ are orthonormal and each is perpendicular to $Z$, $X$, and $N$.
 \end{Lemma}
 \begin{proof}
 Let $E_i\perp Z,X,JX,~for~i-1,...,2n-2$, $E_{2n-1}=Z, E_{2n}=\frac{X-\alpha (X)Z}{\sqrt{1-\alpha^2 (X)}}, N=\frac{JX}{\sqrt{1-\alpha^2(X)}}$ be an orthonormal frame filed.
 \begin{eqnarray*}\Delta f&=&-2(\sum_{i=1}^{2n}g(\nabla_{E_i} JX,E_i)+g(\nabla_NJX,N))\\
 &=&-2(\sum_{i=1}^{2n}g(R(Z,E)X,E)-g(J\phi E_i,E_i)+g(R(Z,N)X,N)-g(J\phi N,N))\\
 &=&2Ricci (Z,X)+2\sum_{i=1}^{2n} g(J\phi E_i,E_i)+2g(J\phi N,N)\\
 &=&2Ricci(Z,X)+2\sum_{i=1}^{2n-2}g(J\phi E_i,E_i)+2\alpha (X)g(N,N)+2g(J\phi N,N)\\
&=&2Ricci(Z,X)+2\sum_{i=1}^{2n-2}g(J\phi E_i,E_i)+2\alpha (X)g(N,N)+2\alpha (X)g(JN,JN)\\
 &=&(2(2n)+4)g(X,Z)+2\sum_{i=1}^{2n-2} g(J\phi E_i,E_i)\\
 &=&(4n+4)g(X,Z)+2\sum_{i=1}^{2n-2}g(J\phi E_i,E_i)
 \end{eqnarray*}
 \begin{equation}\Delta f=(4n+4)g(X,Z)+2\sum_{i=1}^{2n-2}g(J\phi E_i,E_i).\label{laplace}
 \end{equation}\hfill$\qed$
\end{proof}
\section{Double K-contact structures in dimensions 3 and 5}
A transnormal function doesn't have to be isoparametric except in some particular cases. One of these cases is the angle function of double
K-contact structures in lower dimensions.
\begin{Theorem} Let $(M,\alpha ,Z,\beta ,X,g)$ be a double K-contact structure on a closed 3-dimensional or 5-dimensional manifold $M$.
Then the transnormal angle function $f=g(X,Z)$ is isoparametric.
\end{Theorem}
\begin{proof} In dimension 3, Identity (\ref{laplace}) reduces to $$\Delta f=8g(X,Z)=8f.$$ In dimension 5, Identity (\ref{laplace}) reduces to $$\Delta f=12g(X,Z)+\sum_{i=1}^2g(J\phi E_i,E_i)=12f\pm 4$$ since in this case one has $J=\pm \phi$ on the orthogonal complement of $\{Z,X, JX\}$.$\qed$
\end{proof}
\section{Harmonic Unit vector fields in K-contact geometry}
On any Riemannian manifold $(M,g)$, with Ricci tensor $\rho$, one defines the Ricci endomorphism $Q$ by $$\rho (A,B)=g(QA,B).$$ If $(M,g, \alpha ,Z,J)$ is a K-contact structure on $M$, then one has
\begin{equation}QZ=2nZ.\label{ricci}\end{equation} If the K-contact structure is Sasakian one has also the following identity:
\begin{equation}QJ=JQ,\label{commute}\end{equation} that is, the Ricci endomorphism commutes with the transverse complex structure. (See \cite{BLA} for these and more identities on K-contact structures.)
\begin{Lemma}\label{lemma4} On a double K-contact manifold $(M,\alpha ,Z,\beta , X)$, suppose that one of the contact forms, say $\alpha$, is Sasakian. Then one has $\phi J=J\phi$ on the subbundle orthogonal to $\{Z,X, JX=\phi Z\}$. Moreover, $\phi J$ is symmetric and its only eigenvalues are $\pm 1$.
\end{Lemma}
\begin{proof} Let's denote by ${\cal H}$ the tangent sub-bundle orthogonal complement of $\{Z,X,JX\}$. It is easily seen that ${\cal H}$ is $\phi$ and $J$ invariant. The gradient $\nabla f$ of $f=g(X,Z)$ is given by $\nabla f=2JX$.
The Hessian, $Hess_f$ is given by
$$Hess_f(A,B)=(\nabla df)(A,B)=g(\nabla_A\nabla f, B)=2g(\nabla_A(JX),B).$$

Ultimately, using Sasakian identities like $(\nabla_UJ)V=g(U,V)Z-\alpha (V)U$, one obtains:
\begin{equation}Hess_f(A,B)=2g(A,X)g(Z,B)-2\alpha (X)g(A,B)-2g(J\phi A,B).\label{hess}\end{equation}

Since $Hess_f(.,.)$ is a symmetric bilinear form, one deduce from Identity (\ref{hess}) that for any two sections $A$ and $B$ of ${\cal H}$,
$$g(J\phi A,B)=g(J\phi B,A)$$ that is $J\phi$ is a symmetric endomorphism of ${\cal H}$. Moreover, $$g(J\phi A,B)=g(J\phi B,A)=g(B,\phi JA),$$ which means that $\phi$ and $J$ commute on ${\cal H}$.

As a consequence, $$(\phi J)^2=\phi J\phi J=\phi J^2\phi =Id$$ on ${\cal H}$ and hence the only eigenvalues are $\pm 1$. \hfill$\qed$
\end{proof}
\begin{Theorem} Let $(M,\alpha ,Z,\beta , X)$ be a double K-contact structure on a 2n+1 -dimensional closed manifold $M$ with one of the contact forms, say $\alpha$, Sasakian.
Then the function $f=g(Z,X)$ is isoparametric and the vector field $N=\frac{\nabla f}{\|\nabla f\|}$ is a harmonic unit vector field.
\end{Theorem}
\begin{proof}

Isoparametricity holds thanks to the fact that on the distribution orthogonal to \{$Z$, $X$ $JX$\}, the endomorphisms $J$ and $\phi$ commute by Lemma \ref{lemma4}. Using an orthonormal basis of eigenvectors of $J\phi$, Identity (\ref{laplace}) reduces to $$\Delta f=(4n+4)g(X,Z)+2\sum_{i=1}^{2n-2}g(\pm E_i,E_i)=(4n+4)f+2\sum_{i=1}^{2n-2}(\pm 1).$$
Harmonicity follows from the critical point condition for a geodesic unit vector field with integrable orthogonal complement (see Identity (\ref{critical})).

$$E(h)=\rho (E,N)$$ for any $E\perp N$. From the isoparametric condition, $h$ is a function of $f$, hence $E(h)=0$. But also, using identities (\ref{ricci}) and (\ref{commute}), \begin{eqnarray*}\rho (E,N)&=&\rho (E, \frac{JX}{\|JX\|})=\frac{1}{\|JX\|}g(E,QJX)=\frac{1}{\|JX\|}g(E,JQX)\\&=&\frac{2n}{\|JX\|}g(E,JX)=0.\end{eqnarray*}\hfill $\qed$

\end{proof}

\end{document}